\documentclass[11pt]{amsart}
\usepackage{amsmath}
\usepackage{amssymb}
\usepackage{epsfig}
\usepackage{a4wide}

\newcommand{\Z}{\ensuremath{\mathbb{Z}}}
\newcommand{\R}{\ensuremath{\mathbb{R}}}

\newcommand{\Q}{\ensuremath{\mathbb{Q}}}
\newcommand{\X}{\ensuremath{\mathbb{X}}}

\newcommand{\C}{\ensuremath{{\mathbb C}}}
\newcommand{\T}{\ensuremath{{\mathcal T}}}

\renewcommand{\rho}{\varrho}

\newtheorem{thm}{Theorem}[section]

\newtheorem{prop}[thm]{Proposition} 

\newtheorem{defi}[thm]{Definition}

\begin{document}

\title{Substitution Tilings with Statistical Circular Symmetry}

\author{Dirk Frettl\"oh}
\address{Fakult\"at f\"ur Mathematik, Universit\"at Bielefeld, 
Postfach 100131, 33501 Bielefeld,  Germany}
\email{dirk.frettloeh@math.uni-bielefeld.de}
\urladdr{http://www.math.uni-bielefeld.de/baake/frettloe}

\begin{abstract} 
  Two new series of substitution tilings are introduced in which the
  tiles appear in infinitely many orientations. It is shown that
  several properties of the well-known pinwheel tiling do also hold
  for these new examples, and, in fact, for all substitution tilings
  showing tiles in infinitely many orientations. 
\end{abstract} 

\maketitle

\begin{center}
{\em Dedicated to my teacher Ludwig Danzer on the occasion of his 80th
  birthday}
\end{center}

\section{\bf Introduction} \label{intro}
In this article, we introduce two new series of nonperiodic
substitution tilings in 
the plane, where the tiles appear in infinitely many different
orientations. There seems to be a growing interest in such objects,
cf.\ \cite{ors}, \cite{sinv}, \cite{mps}, \cite{bfg}, \cite{yok}. The
standard example is certainly the pinwheel tiling of Conway and Radin
\cite{rad}, and most work about tilings with tiles in infinitely many
orientations is dedicated to this special example.  
It stimulated several re-formulations of concepts
which play a role in the theory of nonperiodic tilings. For instance,
it suggests to define the dynamical system of such a tiling $\T$ as the
closure (in an appropriate topology) of the orbit of $\T$ under the
action of the Euclidean group $E(2)$ of $\R^2$, rather than the
translation group $\R^2$ alone.   

Further examples of tilings showing tiles in infinitely many
orientations are rarely found in the literature.  Sadun
gave a generalization of the pinwheel tiling \cite{sgen}, yielding a
countable number of different substitution tilings with this
property. Apart from this, there are only few examples
known to the author. One is folklore, but nevertheless not widely
known. It is shown in Figure \ref{pw9}. Two more were found by 
Harriss \cite{fh}. The tilings introduced in Sections \ref{sec:pyth} and 
\ref{sec:tipi} of the present article show that the occurrence of
infinitely many orientations in substitution tilings is not
necessarily a rare effect. These tilings include examples of finite
local complexity as well as infinite local complexity (w.r.t.\
Euclidean motions), with an arbitrary number of prototiles, and with
or without primitive substitution matrices. Some of the occurring
substitution factors are Pisot-Vijayaraghavan numbers (including the
smallest one), others are not.  

Section \ref{sec:tilings} states some basic definitions and facts
about substitution tilings. In Section \ref{sec:scs}, we define
statistical circular symmetry of a substitution tiling and prove a
technical result which turns out to be useful in the sequel. 
Section \ref{sec:properties} is dedicated to prove that many
properties of the pinwheel tiling generalise to the whole class of
substitution tilings with tiles in infinitely many orientations. In
particular, these properties are uniform distribution of orientations,
uniform patch frequencies (w.r.t. to the topology used in \cite{ors}
as well as the local rubber topology in \cite{bl}), circular symmetry
of the autocorrelation, and therefore of the diffraction spectrum
\cite{mps}.

\section{\bf Substitution Tilings} \label{sec:tilings}

In the following, $B_r(x)$ denotes the closed ball of radius $r$ around
$x$. A rotation through an angle $\theta$ about the origin
is denoted by $R^{}_{\theta}$. Throughout the text, we will identify
the Euclidean plane with the complex plane, choosing freely the point
of view which fits better to the question at hand.

A {\em tiling} of $\R^d$ is a covering of $\R^d$ with compact sets
--- the {\em tiles} --- which is also a packing of $\R^d$. A tiling
$\T$ is {\em  nonperiodic} if $\T + t = \T$ implies $t=0$. 

Tile substitutions are a simple and powerful tool to generate
interesting  nonperiodic tilings. The basic idea is to give a finite
set of  building blocks --- the {\em prototiles} --- together with a
rule how to enlarge each prototile and then dissect it into copies of
the original prototiles, compare Figures \ref{pw9},
\ref{subs1}, \ref{tipi}. Although the concept applies to arbitrary
dimensions, we restrict ourselves in the following to tilings in the
plane $\R^2$, in order to keep the notation simple. 

Formally, a {\em substitution} $\sigma$ in $\R^2$ is defined for a
collection of prototiles $T^{}_1, \ldots, T^{}_m$ by
$\sigma(T^{}_i) =\{ \varphi^{}_{ijk} (T^{}_j) \,| \,\varphi^{}_{ijk}\in
\Phi^{}_{ij}, j=1 \ldots m \}$, where $\Phi^{}_{ij}$ 
($1 \le i,j \le m$) are sets (possibly empty) of affine maps from
$\R^2$ to $\R^2$. Usually, the maps are of the form $x \mapsto
R^{}_{\alpha} x + t$ for some $\alpha \in [0, 2\pi[$, $t \in \R^2$. 
It is a matter of taste whether one allows also reflections or not. 
For convenience, we will switch between the two concepts. The
particular choice will always be obvious from the context. Two
tiles $T,T' \in \T$ are  {\em of the same type}, if they are congruent
to the same prototile $T_i$. (Sometimes one has to consider tiles
which are congruent but of different types. Then, each tile will get a
label, assigning its type to it. But here, we don't need to consider
such cases.) 
\begin{defi}
A substitution $\sigma$ is called {\em self-similar}, if there is
some $\lambda > 1$ such that for all prototiles $T_i$: 
\[ \lambda T_i = \bigcup_{T \in \sigma(T_i)} T \]
Then, $\lambda$ is called the {\em substitution factor}. 
\end{defi}

Synonyms of substitution factor are {\em inflation factor} or 
{\em length expansion}. 
In the following we consider self-similar substitutions only. 
The substitution $\sigma$ extends in a natural way to all collections
of copies of prototiles: Note that any such collection in the plane
--- and in particular, each tiling ---  
can be represented as $\{ R^{}_{\alpha_1} T^{}_{i_1} +t^{}_1,
R^{}_{\alpha_2} T^{}_{i_2} +t^{}_2, \ldots \}$, where the
$R^{}_{\alpha_i}$ are rotations, $t^{}_i \in \R^2$. The image of 
this set under $\sigma$ is defined as $\{ R^{}_{\alpha_1} \sigma
R^{-1}_{\alpha_1} T^{}_1 + \lambda t^{}_1,  R^{}_{\alpha_2} \sigma
R^{-1}_{\alpha_2} T^{}_2 + \lambda t^{}_2, \ldots \}$, where
$\lambda$ is the substitution factor. In particular, if the set
represents a tiling, this defines the action of $\sigma$ on a tiling. 

For such a tiling $\T=\{ R^{}_{\alpha_1}(T^{}_{i_1})+t^{}_1,
R^{}_{\alpha_2}(T^{}_{i_2})+t^{}_2, \ldots \}$, we assign an angle
$\alpha(T)$ to each tile $T = R^{}_{\alpha} T_i +t$ in $\T$ by
\begin{equation} \label{alpha}
\alpha(T)=\alpha.
\end{equation}   

The following definition helps us to avoid certain pathological cases. 
A {\em patch} is a finite subset of a tiling $\T$, that is, a finite
set of tiles in $\T$. 

\begin{defi} \label{substdef}
Let $\sigma$ be a substitution with prototiles $T^{}_1, \ldots, T^{}_m$ as
defined above. A tiling $\T$ is called {\em substitution tiling} (w.r.t. to
the substitution $\sigma$), if each patch in $\T$ is congruent
to some patch in $\sigma^n(T^{}_i)$ for appropriate $n,i$. \\
The set of all substitution tilings w.r.t. $\sigma$ is called the 
{\em tiling space} $\X_{\sigma}$.  
\end{defi} 

A useful object is the matrix $S^{}_{\sigma}
:=( |\Phi^{}_{ij}| )_{1 \le i,j \le m}$. A substitution $\sigma$ is
called {\em primitive}, if $S^{}_{\sigma}$ is primitive. Recall: 
A nonnegative matrix $M$ is called primitive if there
is some $k \ge 1$ such that $M^{k}$ contains positive entries only.   
By Perron's theorem (\cite{per}, see also \cite{sen}), each primitive
nonnegative matrix has 
a unique eigenvalue which is real, positive, and larger than all other
eigenvalues in modulus. It is not hard to see that the substitution
factor of a primitive self-similar tiling in $\R^2$ is the square root
of the Perron-eigenvalue of the substitution matrix, see for instance
\cite{fre}.  

\section{\bf Statistical circular symmetry} \label{sec:scs}

The following definition helps to simplify terminology. 

\begin{defi}
A substitution tiling $\T$ is called {\em pinwheel-like},
if there are infinitely many different values $\alpha(T)$ for $T \in
\T$, with $\alpha(T)$ defined as in \eqref{alpha}.  
\end{defi}
For a primitive substitution tiling, this is equivalent to requiring
that all copies of each certain prototile occur in infinitely many
orientations in $\T$.  

It is known that the pinwheel tiling fulfils an even stronger
condition, namely, the orientations $\alpha(T)$ of the tiles are
uniformly distributed in $[0,2 \pi[$, see \cite{rad}, \cite{mps}. 
Recall that a sequence $(\alpha^{}_j)_{j \ge 1}$ is called 
{\em uniformly  distributed} in $[0,2\pi[$, if, for all $0 \le x
< y < 2 \pi$,  
\[ \lim_{n \to \infty} \frac{1}{n} \sum_{j=1}^{n} 
  1_{[x, y]}(\alpha^{}_j)  = \frac{x - y}{2 \pi}. \] 
Following \cite{mps}, we will call this property {\em statistical
  circular symmetry}. However, it requires 
some care to define this properly. Since the sum above is not
absolutely convergent, the order of the elements in the sequence
matters. Therefore, one needs to specify how to arrange the tiles 
$T_i \in \T$ in a sequence $(T_i)_{i \ge 0}$. 

\begin{defi}
Let $\T^{}_{\sigma} = \{ T_1, T_2, \ldots \}$ be a primitive
substitution tiling, such that the sequence $(T_j)^{}_{j \ge 1}$
satisfies the following: for all $n \ge 1$, there is some $\ell \ge n$
such that the patch $\{ T_1 \ldots , T_{\ell} \}$ is congruent to
$\sigma^k(T_i)$ for some $k,i$. The tiling
$\T^{}_{\sigma}$ has {\em statistical circular symmetry}, if, for all
$0 \le x < y < 2 \pi$, one has:
\[ \lim_{r \to \infty} \frac{1}{n} \sum_{j=1}^{n} 
  1_{[x, y]}(\alpha(T_j))  = \frac{x - y}{2 \pi}. \] 
\end{defi}

In plain words, the definition ensures that there are infinitely many
$n$ such that $T_1, \ldots, T_n$ are exactly the tiles contained
in some {\em supertile} $\sigma^k(T_i)$. 
Although this definition seems rather technical, it ensures that tiles
(resp.\ the associated angles) are ordered in a natural
way. With respect to the limit, for instance, it is equivalent to
ordering tiles according to their distance from the origin in an
increasing order.  This fact follows from standard properties of
primitive substitution tilings. 

Since the pinwheel tilings are not only pinwheel-like, but also of
statistical circular symmetry, one may argue that the term
'pinwheel-like' is not well chosen. However, Theorem \ref{unifdistr}
below shows the equivalence of pinwheel-likeness and
statistical circular symmetry for primitive substitution tilings, so
the term pinwheel-like becomes obsolete.

The substitution rule shown in Figure \ref{pw9} defines a family of
pinwheel-like tilings. This example is somehow folklore. It was
communicated to the author by L.\ Danzer \cite{dan2}. The substitution
factor is 3, there is only one prototile $T$ --- an isosceles triangle
with edge lengths 1, 2 and 2 --- and the substitution uses only direct
congruences, not reflections. The prototile $T$ is mirror-symmetric
along its vertical axis. Therefore, we equip it with a mark in order
to illustrate the absence of reflections in the substitution rule. The
substituted triangle $\sigma(T)$ contains four copies of $T$ in the
same orientation as $T$, and two copies of $T$ which are rotated by
$\theta=-\arccos(1/4)$.  It is known that $\arccos(1/n) \notin \pi \Q$
for all $n \ge 3$. By Proposition \ref{thetanotinpq} below, the tilings in
this example are pinwheel-like.   

\begin{figure}
\epsfig{file=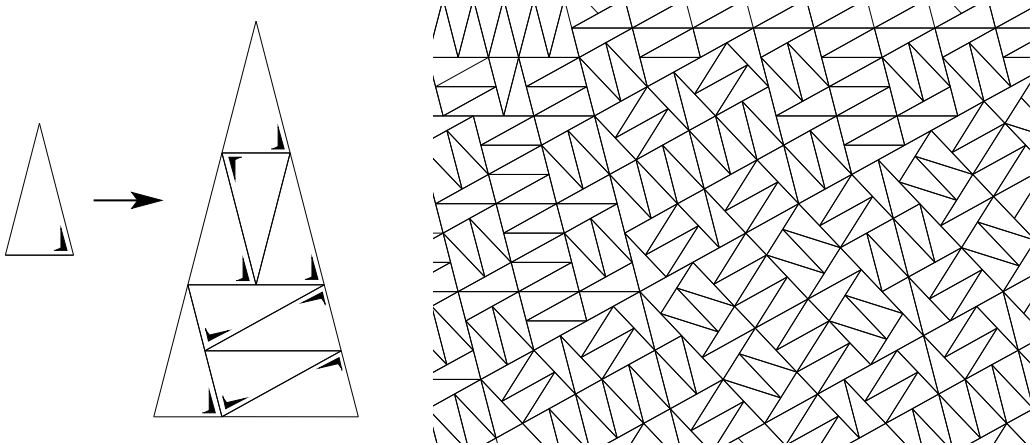}
\caption{A simple substitution rule which generates tilings of
  statistical circular symmetry. \label{pw9}}
\end{figure}

During the last years it, turned out that it is fruitful to consider
tiling spaces or {\em hulls} of tilings rather than particular single
tilings. The {\em hull} of a tiling $\T$ in $\R^d$ is the closure of the
set $\{ \T + x \, | \, x \in \R^d \}$ in the local topology, which is
given by the following metric \cite{so}. Let 
\[ \widetilde{d}(\T,\T') = \inf_{\varepsilon > 0} \{ \varepsilon \, | \, \T
\cap B_{1/\varepsilon}  = R^{}_{\theta} \T' + t, \, \|t\| \le \varepsilon, \,
|\theta| \le \varepsilon \},  \]
and 
\[ d(\T,\T') = \min \{ \frac{1}{\sqrt{2}}, \widetilde{d}(\T,\T') \}. \] 
In plain words, two tilings are close, if they agree on a large ball
(of radius $1/\varepsilon$) around the origin, after a small rotation
$R^{}_{\theta}$ through $\theta$, and after a small translation
$t$. In the case of tilings where all tiles occur only in finitely
many orientations, the rotation part $R^{}_{\theta}$ is usually omitted.
Then, one distinguishes tiles not up to congruence, but up to
translations; the number of prototiles stays finite.  
We should mention that, in the case of tilings with finite local
complexity (see Section \ref{remarks}), this topology is the same as
the local rubber topology \cite{bl} and the local topology
\cite{sch}. For brevity, we don't go into details here.  

The hull of a tiling in $\R^2$, together with the action of the
Euclidean group $E(2) = \R^2 \rtimes O(2)$, is a dynamical system
$\big( \X_{\T}, E(2) \big)$. In this article we consider primitive
substitution tilings only. Note that primitivity ensures that all
types of prototiles occur in some $\sigma^k(T^{}_i)$, for all $i$ and
$k$ large enough. This property is a key for the following theorem,
which is a variant of Gottschalk's theorem, see for instance \cite{rob}. 

\begin{thm}
If $\sigma$ is a primitive substitution, the hull $\X^{}_{\T}$ of each
substitution tiling $\T$ in $\X_{\sigma}$ is $\X_{\sigma}$
itself. Equivalently, the dynamical system $\big( \X_{\T}, E(2) \big)$
is minimal. 
\end{thm}

Consequently, in the case of primitive substitutions, the dynamical
system $\big( \X_{\sigma}, E(2) \big)$ has a unique meaning. 

The following result is needed for the proofs of Theorems
\ref{unifdistr} and \ref{pyth-scs}. One direction ('if') is
well-known, the other direction seems to be new and is necessary for
the proof of Theorem \ref{unifdistr}. 

\begin{prop} \label{thetanotinpq}
Let $\sigma$ be a primitive substitution in $\R^2$ with prototiles
$T^{}_1, \ldots,T^{}_m$. Each substitution tiling $\T^{}_{\sigma}$ is
pinwheel-like, if and only if there are $n,i$ such that
$\sigma^n(T^{}_i)$ contains tiles $T, T'$ of the same type, 
where $\theta = \alpha(T)-\alpha(T') \notin \pi \Q$. 
\end{prop}

\begin{proof}
Without loss of generality, let $\alpha(T)= \theta$ and $\alpha(T')=0$.
Here, $T$ and $T'$ are no mirror images of each other. If there
are prototiles which are mirror symmetric, we break the symmetry with
some markings to avoid ambiguities.  

Since $\sigma$ is primitive, there is some $k$ such that a tile of
type $T^{}_i$ is contained in $\sigma^k(T)$ as well as in 
$\sigma^k(T')$. These two tiles are also rotated against each other
by $\theta$. Thus, $\sigma^{n+k+n}(T^{}_i)$ contains two copies of $T$
which are rotated against each other by $2 \theta$. It follows
inductively that there are copies of $T$ rotated against each other by
$n \theta$ for all $n \ge 0$. Since $\theta \notin \pi \Q$, the
values $n \theta \mod \pi$ are pairwise different. Thus the tiles
of type $T$ occur in infinitely many orientations. Since $\sigma$ is
primitive, the same is true for all prototiles $T^{}_j$.

For the other direction, let $\T^{}_{\sigma}$ be pinwheel-like.
We proceed by showing that if there are infinitely many different
angles $\theta = \alpha(T) - \alpha(T')$ as above, at least one of
them is irrational.
Consider the finitely many angles $\alpha_1, \ldots, \alpha_n$
occurring in the rotational parts of maps in the definition of
$\sigma$. Then, all angles $\alpha(T)$ occurring in $\T$ are linear
combinations  
\[ \big( \sum_{i=1}^n \lambda_i \alpha_i \big) \mod 2\pi \qquad
(\lambda_i \in \Z), \]
in other words, they are elements of the finitely generated
$\Z$-module $\langle \alpha_1, \ldots, \alpha_n \rangle$.   
If all $\alpha_i$ are rational angles, we are done, since then there
are only finitely many such linear combinations mod $\pi$, thus only
finitely many angles $\alpha(T)$. If some $\alpha_i$ are irrational,
it is slightly more complicated:  

In order to get rid of the common factor $\pi$, let
$\beta_i=\frac{\alpha_i}{2 \pi}$. Consider the $\Z$-module $M=\langle
\beta_1, \ldots, \beta_n \rangle_{\Z}$. Since it is finitely generated
and torsion-free (the only element of finite order is 0), there is a
basis $\gamma^{}_1, \ldots, \gamma^{}_k$ of $M$, compare 
\cite[Theorem 10.19]{rot}.  
Let $m$ be the smallest positive integer in $M$. (If there is no such
number, then all differences $\beta_i - \beta_j$ $(j \ne i)$ are
irrational, and we are done.) It has a unique representation $m =
\sum_{i=1}^k \lambda^{}_i \gamma^{}_i$. All other integers $\ell$ in $M$ are
integer multiples of $m$. (Otherwise there is a positive integer
$\lambda \ell + \mu m < m$). Let $q = \mbox{gcd}(\lambda_1, \ldots,
\lambda_k)$. Then $\frac{1}{q}$ is the smallest positive rational
number in $M$, and all others are of the form $\frac{p}{q}$, $p \in
\Z$. In particular, there are only finitely many rational numbers mod
1 in $M$, hence only finitely many rational angles
$\theta = \alpha(T)-\alpha(T') \in \pi\Q$. 

But by the definition of pinwheel-likeness there are infinitely many
such  $\theta$. Thus at least one of them is irrational:
There are two tiles rotated against each other by an
angle $\theta \notin \pi \Q$. By definition of a substitution tiling,
there is a super-tile $\sigma^n(T)$ containing these two tiles, and
the claim follows.  
\end{proof}


\section{\bf The Pythia substitutions} \label{sec:pyth}

\begin{figure}
\epsfig{file=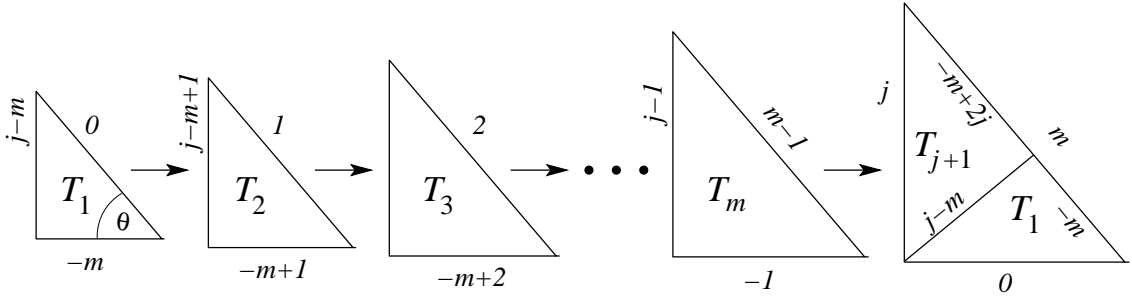, width=150mm}
\caption{The Pythagoras substitution. The edge labels give the edge
  lengths in terms of powers of $\lambda$. For instance, $-m$ is to be
  read as $\lambda^{-m}$. In particular, $0$ means $\lambda^0=1$. 
\label{subs0}}
\end{figure}

First, consider the following family of substitutions (the Pythagoras
substitutions), where the corresponding tilings are {\em not}
pinwheel-like. Let $m \ge 3$ and  
$M$ be the $(m \times m)$-matrix with entries $M^{}_{i+1,i} = 1$ for
$1 \le i \le m-1$; $M^{}_{1,m}=1$; $M^{}_{j+1,m}=1$ for some $1 \le j
\le m-1$, and $M^{}_{i,j}=0$ else. In other words, let $M$ be the
companion matrix of the polynomial $p:=x^m-x^j-1$. For simplicity, we
require $\text{gcd}(m,j)=1$. Then, the primitivity of $M$ can be
shown easily by the methods in \cite{sen}. (If
$\text{gcd}(m,j-1)=q>1$, the substitution is no longer primitive, but
all tilings occurring are already contained in the primitive case.)
Let $\eta$ be the Perron eigenvalue of $M$. By the remark after
Definition \ref{substdef}, $\lambda:= \sqrt{\eta}$ is the substitution
factor of any self-similar substitution with substitution matrix $M$. 

Let $a:=\lambda^{-m}$, and $T_1$ be the orthogonal triangle with
vertices $(0,0),\ (-a,0),\ (-a,\sqrt{1-a^2})$, so its hypotenuse is of
length 1, see Figure \ref{subs0}. This is the first prototile. The
other prototiles are  
$T_{i+1}: = \lambda^i T_1$ for $1 \le i \le m-1$. The substitution is
defined by $\sigma^{}_{m,j} (T_i) = T_{i+1}$ for $1 \le i \le m-1$, and
$\sigma^{}_{m,j} (T_m) = \{ \varphi(T_1), \psi(T_{j+1}) \}$. As indicated
in Figure \ref{subs0}, the substitution acts on $T_m$ by dissecting
$\lambda T_m$ along the altitude on the hypotenuse into tiles of
type $T_1$, $T_{j+1}$. The edge labels in the figure 
indicate the edge lengths in terms of powers of $\lambda$. E.g., $m-1$
is to be read as $\lambda^{m-1}$. Since $\lambda^{2m} =
\lambda^{2j}+1$, it 
follows $1=(\lambda^{j-m})^2 + (\lambda^{-m})^2$ and $\lambda^m =
 \lambda^{j-m+j} + \lambda^{-m}$. The former shows that the triangles
are indeed orthogonal triangles, and inspired the name. The latter
means that the altitude on the hypotenuse indeed dissects
$\lambda T_m$ into $T_1$ and $T_{j+1}$. So the Pythagoras
substitution $\sigma^{}_{m,j}$ is well-defined. 

\begin{figure}
\epsfig{file=konstru.epsi, width=80mm} \qquad 
\epsfig{file=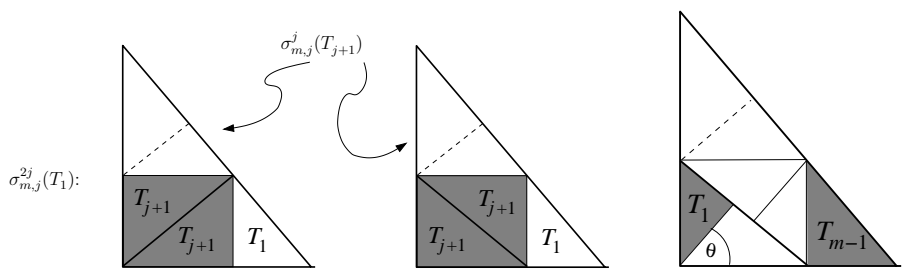,  width=30mm} 
\caption{The construction of the Pythia substitution $\rho^{}_{m,j}$
  out of the Pythagoras substitution $\sigma^{}_{m,j}$. Left: the
  $2j$-th iterate of $T_1$ under $\sigma^{}_{m,j}$. Centre: The same,
  with two tiles flipped. This patch defines the first iterate of
  $T_1$ under $\rho^{}_{m,j}$. Right: The first iterate of $T_{m-j+1}$,
  used in the proof of Theorem \ref{pyth-scs}. \label{konstr}}
\end{figure}

In order to obtain a pinwheel-like substitution, consider
$(\sigma^{}_{m,j})^{2m}(T_1)$ (see Figure \ref{konstr}, left). One
diagonal of the grey rectangle, consisting of two tiles $T_{j+1}$, is
the altitude on the hypotenuse of the large triangle. The new
substitution $\rho^{}_{m,j}$ arises by choosing the other diagonal 
of the rectangle, Figure \ref{konstr} (centre). This defines the
substitution $\rho^{}_{m,j}(T_1)$ of $T_1$. The concrete maps used by
$\rho^{}_{m,j}(T_1)$ can be obtained from the figure and the
corresponding maps in $\sigma^{}_{m,j}$. The substitution for the
other prototiles is defined by 
\begin{equation} \label{defrho}
\rho^{}_{m,j}(T_i) := (\sigma^{}_{m,j})^{i-1} \big(
\rho^{}_{m,j}(T_1) \big), \quad 2 \le i \le m. 
\end{equation}
Since $\sigma^{}_{m,j}$
is well-defined, $\rho^{}_{m,j}$ is, too. Since the substitution
factor of the Pythagoras substitution $\sigma_{m,j}$ is $\lambda$, and
the Pythia substitution arises from $2m$ iterations of
$\sigma_{m,j}$, the substitution factor of 
the Pythia substitution $\rho^{}_{m,j}$ is $\lambda^{2m}$. The
particular Pythia substitution $\rho^{}_{3,1}$ is shown in
Figure \ref{subs1}. The proof that all Pythia substitutions
$\rho^{}_{m,j}$ generate pinwheel-like tilings is given in Theorem
\ref{pyth-scs}.

\begin{figure}
\epsfig{file=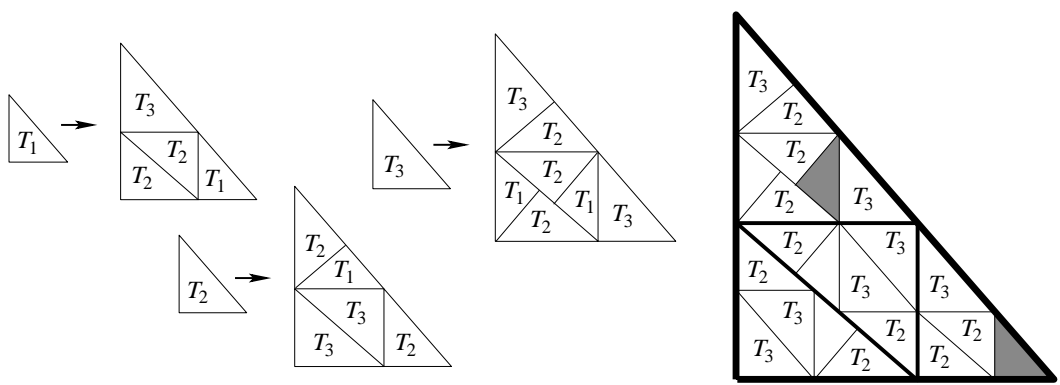,  width=140mm}
\caption{The substitution rule for the Pythia substitution
  $\rho^{}_{3,1}$ (left), and the second iterate of $T_1$ (right). The
  two grey shaded tiles are both of type $T_1$, rotated against each
  other by an angle $\theta \notin \pi \Q$. \label{subs1}}
\end{figure}

\section{\bf The tipi substitutions} \label{sec:tipi}

Let $\eta^{}_{m,j}$ be the root of $x^m-x^{2j}-2x^j-1$, $3 \le m$,
$1 \le j < m/2$, $a=a^{}_{m,j}=(\eta^{}_{m,j})^j$,
$\theta:=\arccos(1/2a)$. 
As above, $m$ is always the number of prototiles of the
tiling under consideration. The first prototile, $T^{}_0$, is an
isosceles triangle with edges of length $1,a,a$. To be explicit,
let $T^{}_0$ be the triangle with vertices $0,\, 1,\, \frac{1}{2} + i
\sqrt{a^2- \frac{1}{4}}$. The other prototiles are defined recursively by
$T^{}_i:= \eta^{}_{m,j} T^{}_{i-1}$ ($1 \le i \le m-1$).
The substitution rule is $\sigma(T^{}_i):=T^{}_{i+1}$ for $0 \le i \le
m-2$, and $\sigma(T^{}_m):=\{\varphi^{}_0(T^{}_0),\varphi^{}_1(T^{}_j),
\varphi^{}_2(T^{}_j), \varphi^{}_3(T^{}_{2j})\}$; see Figure
\ref{subs1} for the case $j=1, \, m=3$.  The Euclidean motions
$\varphi^{}_i$ can be read off from the figure. Explicitly, they read 
\[ \varphi^{}_0(z)=z+a^2, \, \varphi^{}_1(z)=e^{2 \pi i \theta} z+a^2, \,
\varphi^{}_2(z)= e^{2 \pi i \theta} \overline{z}, \varphi^{}_3(z)=z+
\frac{1}{2} +i\sqrt{a^2- \frac{1}{4}}.   \]
Note that one reflection is involved, expressed by the complex
conjugation in $\varphi^{}_2$.

\begin{figure}
\epsfig{file=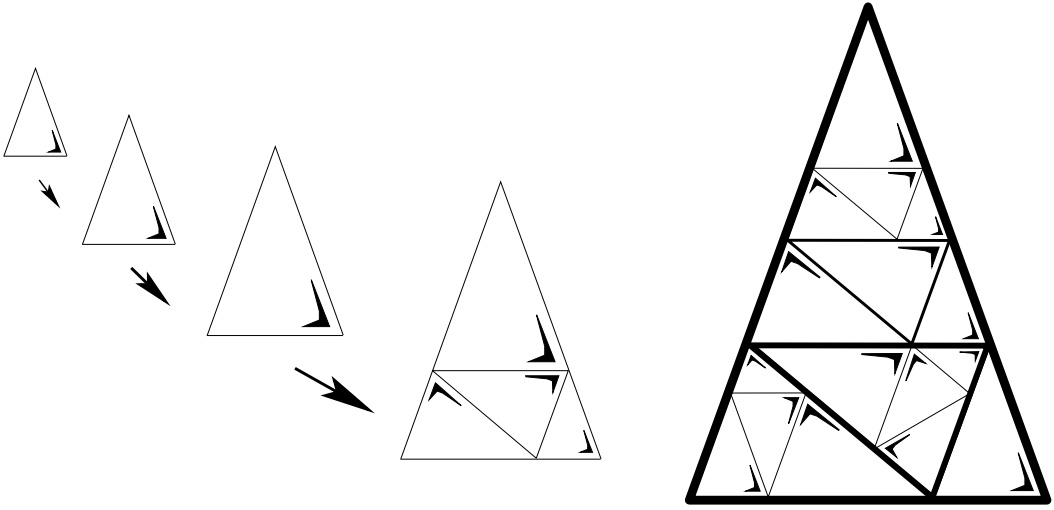, width=140mm}
\caption{The tipi substitution for the case $m=3,\, j=1$. The shape of
  the tiles resembles a tipi, a conical tent used by Sioux and other
  tribes in the Great Plains, thus the name. \label{tipi}}  
\end{figure}

\section{\bf Properties of the Tilings} \label{sec:properties}

The uniform distribution of orientations was shown for the particular
case of the pinwheel tiling in \cite{rad2}, see also \cite{mps}. The
present proof is just a generalisation to the arbitrary case. It uses
Perron's theorem \cite{per}, Weyl's criterion and Proposition
\ref{thetanotinpq}.  

\begin{thm} \label{unifdistr}
Let $\T^{}_{\sigma}$ be a pinwheel-like substitution tiling,
where $\sigma$ is a primitive substitution in $\R^2$ with
prototiles $T^{}_1, \ldots, T^{}_m$. Then $\T^{}_{\sigma}$ is of
statistical circular symmetry. Consequently, $\X^{}_{\sigma}$ is of
statistical circular symmetry, too. 
\end{thm}

\begin{proof}
Weyl's criterion states that $\{ \alpha^{}_j \}_{j \in \Z} = \{e^{i
  \varphi_j}\}_{j \in \Z}$ is uniformly distributed in $[0,2 \pi[$, iff 
\[ \lim_{n \to \infty} \frac{1}{n} \sum_{j=1}^{n} (\alpha_j)^t = 0 \]
for all $0 \ne t \in \Z$. So, for $t \in \Z$, consider the matrix 
\[ M(t) = \biggl( \sum_{j=1}^{S^{}_{k \ell}} e^{i \varphi_j t}
\biggr)_{k \ell} \quad (1 \le k, \ell \le m), \]
where $S$ denotes the substitution matrix and, for $k,
\ell$ given, $\varphi_j$ denotes the orientation of the $j$-th tile of
type $k$ in $\sigma(T_{\ell})$. Thus, $\big( M(1) \big)_{k \ell}$
contains the sum of the orientations of the tiles of type $k$ in
$\sigma(T_{\ell})$. Inductively, it follows that 
$\big( M(1)^r \big)_{k \ell}$ contains the sum of the orientations of
the tiles of type $k$ in $\sigma^r(T_{\ell})$. We proceed to prove 
\[ \lim_{r \to \infty} \frac{(M(t)^r)_{k \ell}}{(S^r)_{k \ell}}=0,\] 
from which uniform distribution follows. 
Without loss of generality, let one occurring orientation be $0$. (The
substitution rule can 
always be modified in order to achieve that, by adding an appropriate
general rotation.) Let $r \ge 1$. Assume 
\[ |\big( M(t)^r \big)_{k \ell}| = |\sum_{j=1}^{(S^r)_{k \ell}} e^{i
  \varphi_jt}| = S_{k \ell}, \]
for some $t >0$. Then, $\varphi_j t \in \pi \Z$ for all $\varphi_j$
considered (that is, for all $\varphi_j$ associated with some $T_k$ in
$\sigma^r(T_{\ell})$). But, by Proposition \ref{thetanotinpq}, 
not all occurring angles $\varphi_j$ are elements of $\pi \Q$. Thus,
for all $r,t \ge 1$, there is $(k,\ell)$ such that 
\[ | \sum_{j=1}^{(S^r)^{}_{k \ell}} e^{i \varphi_j t} | <
(S^r)_{k \ell}. \] 
Let $\lambda$ be the substitution factor of the substitution (so,
$\lambda^2$ is the Perron-Frobenius eigenvalue of $S$), and
let $\eta$ be the Perron-Frobenius eigenvector of $A:= \big( | (M(t)^r
)_{k \ell}| \big)_{k \ell}$. By Perron's  theorem, 
$0< \lim_{r \to  \infty}  \frac{(S^r)_{k \ell}}{(\lambda^2)^r} =
c^{}_{k \ell}$, with a constant $0< c^{}_{k\ell} < \infty$. (More
precisely, each vector $(c^{}_{k1}, \ldots, c^{}_{km})^T$ is a right
eigenvector of $S$ for the Perron-Frobenius eigenvalue $\lambda^2$ of
$S$, and each vector $(c^{}_{1 \ell}, \ldots, c^{}_{m \ell})$ is a left
eigenvector for $\lambda^2$, and all vectors are positive.)
Analogously, $0< \lim_{r \to \infty} 
\frac{(A^r)_{k \ell}}{\eta^r} = a^{}_{k \ell} < \infty$.
Since $A \le S^r$ and $A \ne S^r$, it follows $\eta < \lambda^2$,
again by Perron's theorem. Thus, 
\[ \left| \frac{(M(t)^r )_{k \ell}}{(S^r)_{k \ell}} \right| \le
\frac{(A^r)_{k \ell}}{(S^r)_{k \ell}} = \frac{(A^r)_{k
    \ell}}{\eta^r} \frac{\eta^r}{(S^r)_{k \ell}} \le c
\left(\frac{\eta}{\lambda^2} \right)^r \to 0 \quad(r \to \infty). \] 

The claim for the hull $\X_{\sigma}$ follows immediately, since each
element of $\X^{}_{\sigma}$ shows already statistical circular
symmetry.  
\end{proof}

\begin{thm} \label{pyth-scs}
Let $m \ge 3$, and $1 < j <m$. All tilings defined by the Pythia
substitution $\rho^{}_{m,j}$ are of statistical circular symmetry. 
\end{thm}
\begin{proof}
Utilising Proposition \ref{thetanotinpq}, we need to show that there
are two tiles of the same type $T_i$ in some $(\rho^{}_{m,j})^k(T_j)$
which are rotated against each other by an angle $\theta \notin \pi\Q$. 

Consider the Pythia substitution with respect to chirality of the
tiles. Let us call the tiles in Figure \ref{subs0} (without the two 
rightmost ones) right-handed, denoted by $T_i$ as usual; and let us
call their mirror images --- obtained by reflection in the horizontal
axis --- left-handed, denoted $\overline{T}_i$. 
The following properties follow from the construction of the Pythia
substitution $\rho^{}_{m,j}$. 

1. For each $T_i$, $\rho^{}_{m,j}(T_i)$ contains a translate of
$T_{i}$, rotated by $0 \mod \frac{\pi}{2}$. \\
2. $\rho^{}_{m,j}(T_1)$ contains a translate of $\overline{T}_{j+1}$,
rotated by $0 \mod \frac{\pi}{2}$. \\
3. $\rho^{}_{m,j}(T_{m-j+1})$ contains a translate of $T_1$, rotated
by $\theta \mod \frac{\pi}{2}$. \\ 
4. $\rho^{}_{m,j}(T_{m-j+1})$ contains a translate of
$\overline{T}_{j+1}$, rotated by $-\theta \mod \frac{\pi}{2}$. 

The first three statements are immediate consequences of the
construction, see Figure \ref{konstr}. For the fourth one, recall that
the $m$-th iterate of $T_1$ under the Pythagoras substitution
$(\sigma^{}_{m,j})^m$ contains a translate of $\overline{T}_1$,
rotated by $-\theta \mod \frac{\pi}{2}$, see Figure
\ref{subs0}. Consequently,  $(\sigma^{}_{m,j})^m(T_{j+1})$ contains a
translate of $\overline{T}_{j+1}$, 
rotated by $-\theta \mod \frac{\pi}{2}$. By definition,
$\rho^{}_{m,j}(T_{m-j+1})$ contains a translate of
$(\sigma^{}_{m,j})^{j+m-j}(T_{j+1}) = (\sigma^{}_{m,j})^m(T_{j+1})$
(compare Figure \ref{konstr} and \eqref{defrho}), thus it contains a
translate of $\overline{T}_{j+1}$, rotated by $-\theta \mod
\frac{\pi}{2}$. 

Altogether, by 3.\ and 2., $(\rho^{}_{m,j})^2(T_{m-j+1})$ contains a
translate of $\overline{T}_{j+1}$, rotated by $\theta \mod
\frac{\pi}{2}$. By 4.\ and 1., $(\rho^{}_{m,j})^2 (T_{m-j+1})$
contains a translate of $\overline{T}_{j+1}$, rotated by $-\theta \mod 
\frac{\pi}{2}$.  By Proposition \ref{thetanotinpq}, it remains to show
that $\theta - (-\theta) = 2 \theta \notin \pi \Q$. 

We proceed by showing that $\theta \notin \pi \Q$. Consider the tile
$T_1$ embedded in $\C$, with vertices 
$0,\ -\lambda^{-m},\ -\lambda^{-m} + i\lambda^{j-m}$, see Figure
\ref{subs0}. Then $\theta \notin \pi \Q$ iff $-\lambda^{-m} +
i\lambda^{j-m}$ is not a complex root of unity, or equivalently,
$\lambda^{-m} \ne \cos(\frac{k   \pi}{n})$ 
for all $k, n \in \Z$. Recall that $\lambda^2$ is a root
of $x^m-x^j-1$, so $\lambda$ is a root of $p:=x^{2m}-x^{2j}-1$. It is
well known that a polynomial in $\Z[x]$ which is reducible in $\Q(x)$
is already reducible in $\Z[x]$. Thus, no non-integer coefficients can
occur in the factorisation of $p$, wherefore the prime polynomial of
$\lambda$ is of the form $x^{\ell} \pm \cdots \pm 1$. Consequently,
$\lambda$ is an algebraic integer, as well as a unit. 
Consequently, $\lambda^{-1}$ is an algebraic integer as well, as is
$\lambda^{-m}$.

Assume that $\lambda^{-m} = \cos(\frac{k \pi}{n})$ for some $k,n \in \Z$,
where $\text{gcd}(k,n)=1$. Since $\lambda^{-m} \notin \{-1,0,1\}$, we can
exclude $k=0,\ n=1$ and $n=2$. Let $\xi^{}_n = e^{2 \pi i \frac{k}{n}}$ be a
primitive $n$-th root of unity. Then, $\lambda^{-m} = \frac{1}{2}
(\xi^{}_n + \overline{\xi^{}_n})$. 


On the other hand, it is known that the ring of integers in
$\Q(\xi^{}_n+\overline{\xi^{}_n})$ equals $\Z[\xi^{}_n +
\overline{\xi^{}_n}]$ \cite[Prop. 2.16]{wash}, and all integers in
$\Q(\xi^{}_n + \overline{\xi^{}_n})$ are of the form 
$\sum_{k=0}^{n-1} \beta_k
(\xi^{}_n + \overline{\xi^{}_n})^k$, where $\beta^{}_k \in \Z$. 
But the unique representation of $\lambda^{-m}$ is $\frac{1}{2}
(\xi^{}_n + \overline{\xi^{}_n})$, so $\lambda^{-m}$ is not an
algebraic integer, which is a 
contradiction. Therefore, $\lambda^{-m} \ne \cos(\frac{k \pi}{n})$ for
all $k,n \in \Z$, which proves the claim.
\end{proof}

Note that the relevant angle for the tipi substitution is $\varphi$ with
$\cos(\varphi)= \frac{1}{2} (\eta_{m,k})^{-j}$, and $(\eta_{m,k})^m =
(\eta_{m,k})^j+1$. So the above argument fails in this case: it may be
that $\eta_{m,k} = 2 \cos(\frac{k \pi}{n})$. We don't see a similarly
simple argument for all tipi tilings being of statistical circular
symmetry. Nevertheless, each particular case can be easily checked
whether it is. All cases checked so far are of statistical circular
symmetry.  

\begin{thm} \label{upf}
Every primitive substitution tiling $\T_{\sigma}$ of statistical
circular symmetry shows uniform patch frequency, in the
following sense: For every $\varepsilon >0$ and each finite patch 
$P \subset \T$ there
is $r >0$ such that every ball of radius $r$ contains a
translate of $P$, up to a small rotation $R^{}_{\alpha}$, where
$|\alpha| < \varepsilon$.
\end{thm}

\begin{proof}
Let $\varepsilon > 0$, and let $P$ be some patch in $\T$. By
Definition \ref{substdef} there is $k \ge 1$ such that a copy of $P$
is contained in some supertile $\sigma^k(T_i)$. By primitivity, there
is  $M > k$ such that a copy of $P$ is contained in all supertiles
$\sigma^m(T)$, $m \ge M$. 

Let $T_j$ be some prototile. 
By the proof of Proposition \ref{thetanotinpq}, there is $\theta
\notin \pi \Q$ such that holds: For all $N \ge 0$ there is 
$\ell \ge 0$ such that each supertile $\sigma^{\ell}(T)$ contains a
translate of $R_{n \theta} T_j$ for all $0 \le n \le N$. 

Every angle $\beta \in [0, 2 \pi[$ can be approximated by $n \theta
\mod \pi$ (with suitable $n \ge 0$) arbitrarily close. 

Let us combine these observations. There is a particular supertile
$\sigma^{\ell}(T_i)$ containing translates of $R_{n \theta} T_j$ for
all $0 \le n \le N$. Thus, there is $s > 0$ such that all supertiles
$\sigma^{\ell + s}(T)$ contain translates of $R_{n \theta} T_j$ for
all $0 \le n \le N$. Then, all supertiles $\sigma^{\ell + s + m}(T)$
contain copies of $P$, in angles $0, \theta, \ldots, N \theta$. The
orientation of $P$ can be approximated arbitrarily close by increasing
$m$. Thus, there are $\ell, s, m$ such that each supertile
$\sigma^{\ell + s + m}(T)$ contains a copy $P'$ of $P$, with $P' =
R_{\alpha} P - t$, $t \in \R^2$, $|\alpha| < \varepsilon$. 
For this particular choice $\ell, s, m$, there is $r'>0$ such that each
supertile $\sigma^{\ell + s + m}(T)$ fits into a ball with radius
$r'$. Thus each ball of radius $r = 3r'$ contains an entire
supertile, which proves the claim.
\end{proof}

\section{Remarks} \label{remarks}

The circular symmetry of $\X_{\sigma}$ implies the circular symmetry
of the autocorrelation (compare for instance \cite{mps}, \cite{bfg})
of its elements, and therefore the circular symmetry of its
diffraction spectrum. This has been known for the pinwheel
tilings. The results in this article imply the circular symmetry of the
diffraction of all primitive substitution tilings with tiles in
infinitely many orientations. 

The relevance of Theorem \ref{upf} relies on the connection to
dynamical systems of a tiling space. In the case of finitely many
orientations (that is, the number of $\alpha(T)$ as in \eqref{alpha}
is finite), the dynamical system $(\X_{\sigma}, \R^2)$ is uniquely
ergodic, if and only if the tilings in $\X_{\sigma}$ have uniform
patch frequency \cite{lms}. This result plays a central role for
further investigations of such systems. The generalization of this
result to tilings with statistical circular symmetry will be worth
investigating in the future. 

A tiling is of {\em finite local complexity} (FLC), if for each $r>0$,
the number of congruence classes of $\T \cap B_r(x)$ ($x \in \R^d$) is 
finite. The Pythia and the tipi substitution tilings introduced here
show both cases: FLC and non-FLC. This can be shown by the methods
in \cite{dan}, \cite {pfr}. 
Note that, in the case of finitely many orientations, FLC is
frequently defined by 'for each $r>0$, the number of {\em translation}
classes of $\T \cap B_r(x)$ ($x \in \R^d$) is finite'. In the present
context we deviate from this convention for obvious reasons. 

The substitution factor of a self-similar substitution determines
already many properties of the corresponding tilings. For
instance, if the substitution factor is a PV number, then the tiling
is FLC under a certain (mild) condition \cite{pfr}. PV number stands for
Pisot-Vijayaraghavan number, an algebraic integer which algebraic
conjugates are all smaller than one in modulus. The substitution
factors of the 
tilings in this article cover various cases. In particular, they are of
arbitrary algebraic degree $m \ge 2$. For $m \ge 3$, this can be seen
by the irreducibility of $x^m-x-1$. For $m=2$, one needs to alter the
setting slightly: The case $m=4,j=2$ does not obey the requirement
gcd$(m,j)=1$. Therefore, the Pythagoras substitution for these values 
is not longer primitive: there are four prototiles $T_1, T_2, T_3,
T_4$, but in each Pythagoras tiling to $\sigma^{}_{4,2}$, only two of
them occur, either $T_1, T_3$, or $T_2, T_4$. The substitution
$(\sigma^{}_{4,2})^2$ does the job: It uses only two prototiles $T_1,
T_3$. The substitution factor is the golden mean $\tau =
\frac{\sqrt{5}+1}{2}$. The corresponding Pythia tiling (see 'golden
pinwheel tilings' in \cite{fh}) has substitution factor $\tau+1$, which
is a quadratic PV number. Other PV numbers occurring as substitution
factors for Pythagoras tilings are the dominant roots of  $x^3-x-1$
(which is the smallest PV number among all algebraic integers \cite{sie}),   
of $x^3-x^2-1$ and of $x^4-x^3-1$. Thus, PV numbers which are
substitution factors of the Pythia tilings include powers of those. 
PV numbers occurring as substitution factors for tipi tilings include
the ones mentioned for the Pythagoras tilings, as well as
the dominant root of $x^3-2x^2+x-1$ (in the case $m=5,j=2$). 

\section*{Acknowledgements}
It is a pleasure to thank Michael Baake and Christian Huck for helpful
discussions. This work was supported by the German Research Council
(DFG) within the Collaborative Research Centre 701.

\end{document}